\documentclass{amsart}

\usepackage{amsmath,amsthm,amsfonts,tikz,bm}

\usepackage[english]{babel}
\usepackage{amssymb,fontenc}
\usepackage{latexsym,wasysym,mathrsfs}
\usetikzlibrary{arrows}
\usepackage{cite}
\newtheorem{theorem}{Theorem}[section]
\newtheorem{lemma}[theorem]{Lemma}

\theoremstyle{definition}
\newtheorem{definition}[theorem]{Definition}

\theoremstyle{remark}

\numberwithin{equation}{section}
\newcommand{\R}{{\mathbb R}}
\newcommand{\N}{{\mathbb N}}

\def\loc{{\rm loc\,}}
\def\supp{{\rm supp\,}}
\def\ds{\displaystyle}
\numberwithin{equation}{section}

\def\ol{\overline}

\usepackage{graphicx}%
\usepackage{multirow}%
\usepackage{amsmath,
amssymb,amsfonts}%

\title[Laplace equation in weighted Sobolev spaces]{Global solvability of the Laplace equation in weighted Sobolev spaces}

\author[B.T. Bilalov]{Bilal T. Bilalov}
\address{Institute of Mathematics and Mechanics\\
 Azerbaijan\\
 e-mail:  b\_bilalov@mail.ru}

\author[N.P. Nasibova]{Natavan P. Nasibova}
\address{Azerbaijan State Oil and Industry University\\
 e-mail: natavan.nasibova@asoiu.edu.az}

\author[L.G. Softova]{Lubomira G. Softova}
\address{Department of Mathematics
University of Salerno, Italy\\
 e-mail:  lsoftova@unisa.it}

\author[S. Tramontano]{Salvatore Tramontano}
\address{Department of Mathematics
University of Salerno, Italy\\
 e-mail: stramontano@unisa.it}



\begin{document}

\begin{abstract}

We consider a non-local boundary value problem for the  Laplace  equation in unbounded strip,  studding the weak  and strong solvability  of that problem in the framework of the weighted Sobolev space $W^{1,p}_\nu$ with a  Muckenhoupt weight. Making use of non-harmonic analysis tools, we proved that any weak solution belonging  to  $W_{\nu}^{2,p}$  is also a strong solution and satisfies the corresponding boundary conditions.  It should be noted that such problems do not fit into the general theory of elliptic equations and require a particular approach. 
\end{abstract}

\maketitle

\bigskip


\section{Introduction}

The classical existence and regularity  theory for linear  PDEs  
leave untreated  a lot of problems, arising in the mechanics and the mathematical physics.    An example  of such a model problem is the following degenerate elliptic equation, studied by a Moiseev in \cite{6}.  More precisely,   he considered the  following degenerate equation in infinity strip:
\begin{equation} \label{eq-1-1}
\begin{cases}
y^m u_{xx} +u_{yy}=0 & (x,y)\in(0,2\pi)\times(0,\infty)\\
u(x,0)=f(x) &x \in (0,2\pi) \\
u(0,y)=u(2\pi,y) &y \in (0,\infty)\\
u_x (0,y)=0 & y\in (0,\infty)
\end{cases}
\end{equation}
with $m>-2$ and $f\in C^{2,\alpha}[0,2\pi]$, where  $\alpha\in (0,1).$  This problem is non-local and  the boundary conditions are given on semi infinite lines. 
Under  the natural assumption  of boundedness of the solution at infinity, the author obtained existence and uniqueness of aclassical solution and an explicit integral representation for it, that permits to release the regularity assumptions on $f.$

 Similar boundary  problems but  for mixed type equations  are investigated  by Frankl in his study  on  the transonic flow around symmetric airfoils (see \cite{7,8}). 
 More results on existence of classical solutions of  problem \eqref{eq-1-1} are obtained in  \cite{9}, in the case of uniformly  elliptic equation, and in \cite{10} for multidimensional parabolic equation.

We start our studies  with the case $m=0$, while the degenerate problem is an object of further research.
Our goal is twofold, to obtain strong/weak solvability  of \eqref{eq-1-1}  and to study the regularity of that solutions in new function spaces.  Our interest is pointed on  weighted Lebesgue spaces  $L^p_\nu, $ where $p\in(1,\infty)$ and the  weight function $\nu$ belongs to the Muckenhoupt class $A_p.$ Since we study  problem \eqref{eq-1-1}  in unbounded, with respect to $y,$ domain we suppose that the weight depends only on $x.$ So the Sobolev space builded upon consists of functions heaving distributional derivatives in some weighted  $x$-space and  integrable with respect to $y.$  

Starting with existence result for weak solutions in $W^{1,p}_\nu,$ we show that under suitable boundary conditions this weak solution is strong too.

Let us note,  that this  problem cannot be treated with the classical methods developed for linear elliptic operators. 
Our technique is based on the spectral theory and the approach developed in \cite{6}.
Precisely, we make use  of biorthonormal systems and Fourier series technique  in Banach function  spaces, see Duffin and Schaeffer \cite{DS}. This approach permits  to extend   harmonic analysis’ methods beyond the Hilbert spaces  in  hopes of applying   the   Fourier series methods  (see also \cite{Ch,Ol,Z}).

The present paper extends some  results obtained in \cite{6,9,11},  passing from classical to generalized  solutions.

In the following we use  the standard notation:
\begin{itemize}

\item $\R_+=(0,\infty)$ and $\N_0=\N \cup \{0 \}$;

\item $\Pi=(0,2\pi)\times(0,\infty)$ is unbounded strip in $\R^2$ with boundary 
$
\partial \Pi= J_0 \cup \overline{J} \cup  J_{2\pi},
$
where
\begin{align*}
J&=\{ x\in (0,2\pi), \ y=0\},\\
J_0 &=\{x=0,  \  y\in(0,\infty)\},\\
J_{2\pi } &=\{x=2\pi, \ y\in(0,\infty) \};
\end{align*}

\item $C_{2\pi}^{\infty}(\ol J)=\{\eta(x) \in C^{\infty} ([0,2\pi])  : \   \eta(2\pi)=0\};$

\item the letter $C$ indicates a positive constant, which value varies from  line to line. 
\end{itemize}

\section{Auxiliary  results}
Consider  the following set of {\it test functions}
\begin{equation}\label{eq-1}
\begin{split}
C_{J_0}^{\infty}(\overline{\Pi})=\Big\{&\varphi \in C^{\infty} (\overline{\Pi}):  \ \varphi |_{J\cup J_{2\pi}}=0, \quad  \exists \, \xi_{\varphi} >0 : \\
&\varphi (x,y)= 0 \quad \forall\, (x,y)\in [0,2\pi] \times [\xi_{\varphi},\infty) \Big\}. 
\end{split}
\end{equation}

Let $\nu :(0,2\pi)\to [0,\infty]$ be a  weight function, $\nu\in L^1(0,2\pi)$,  and  $|\nu ^{-1} (\{0;\infty\})|=0.$
We consider the weighted Lebesgue space $L^p_\nu (\Pi)$, $p\in (1,\infty)$, endowed by the norm  
$$
\| u\|_{L^p_\nu(\Pi)} =\int_0^{\infty}\Big(\int_0^{2\pi}|u(x,y)|^p \, \nu (x) \, dx \Big)^{\frac{1}{p}}\,dy.
$$
The corresponding $\nu$-weighted Sobolev space $W^{m,p}_\nu (\Pi)$ is defined as the set of all measurable functions, having distributional  derivatives up to order $m$ in the space $L^p_\nu(\Pi)$ for which the following norm is finite 
$$
\| u\|_{W^{m,p}_\nu (\Pi)} =\sum_{0\leq |\alpha |\leq m}\| D^{\alpha} u\|_{L^p_\nu (\Pi)}.
$$
Take  $p\in (1,\infty)$, by $L^p_\nu(0,2\pi)$ and $W_{\nu}^{m,p}(0,2\pi)$, we denote the Lebesgue and Sobolev spaces over the interval $(0,2\pi)$, endowed by the respective norms
$$
\| f\|_{L^p_\nu(0,2\pi)}= \Big (\int_0^{2\pi} |f(x)|^p \nu(x)\, dx 
\Big )^{\frac{1}{p}}, \qquad
\| f\|_{W^{m,p}_\nu(0,2\pi)} =\sum_{k=0}^m \| f^{(k)} \|_{L^p_\nu(0,2\pi)}.
$$
 For  completeness of the exposition, we define the Muckenhoupt class $A_p(0,2\pi)$ of $2\pi$-periodic weights $\nu$ that  satisfy the  condition
\begin{equation}\label{eq-Muck}
\sup_{I\subset \R} \Big (\frac{1}{|I|} \int_I\nu(x)\, dx \Big) \Big (\frac{1}{|I|} \int_I \nu(x)^{-\frac{1}{p-1}} \, dx \Big )^{p-1}=[\nu]_p <\infty,
\end{equation} 
where the supremum    is taken over all bounded  intervals $I\subset \R$ and $[\nu]_p$ is the Muckenhoupt constant of $\nu.$  An immediate consequence of the definition \eqref{eq-Muck} is that  if $\nu \in A_p(0,2\pi)$ then  $\nu,\nu^{-1/(p-1)}\in L^1_\loc(\R).$    Moreover,
the following properties hold (cf. \cite{13}).

\begin{lemma}\label{lem-RH}   
    Let $\nu\in A_p(0,2\pi),$ $p\in(1,\infty)$ 
    \begin{itemize}
        \item {\em Inclusion property:}    there  exists $q\in (1,p)$ such that $\nu \in A_q(0,2\pi);$ 

    \item {\em Reverse H\"older inequality:}
     there exists  $\delta >0$ depending only on $p$ and  $[\nu]_p,$ such that
    $$
\Big( \frac{1}{|I|} \int_I \nu(x)^{1+\delta}\, dx  \Big)^{\frac{1}{1+\delta}}\leq \frac{C_\delta}{|I|}\int_I \nu(x)\, dx
    $$
  for each $I,$  where the  constant $C_\delta$ does not depend  on $I.$ 
       \end{itemize}
\end{lemma}
For our goals we assume in addition that the weight satisfies the condition $\nu(x)=\nu(2\pi-x)$ for a.e. $x\in(0,2\pi).$

The next lemma gives a characterisation  of the weighted spaces that we are going to use and the proof follows  easily from \cite[Lemma~2.6]{25}.
\begin{lemma}    \label{L2.1} 
Let $\nu \in  A_p(0,2\pi),$ with $1<p<\infty$. Then
 \begin{enumerate}
\item $L^p_\nu(0,2\pi)\subset L^1(0,2\pi)$ continuously embedded; 
\item $\overline{C_0^{\infty} (0,2\pi)}=L^p_\nu (0,2\pi)$ and the closure is taken with respect to the norm  in $L^p_\nu(0,2\pi).$
\end{enumerate}
\end{lemma}

In what follows we need the  Young-Hausdorff inequality  related to  the classical system of   functions $\{1,\cos nx, \sin nx\}_{n\in\N}.$

\begin{theorem}[Young-Hausdorff \cite{1}]\label{thmYH}
Let $f\in L^p(0,2\pi),$   $ 1<p\leq 2 $ and  consider   for $n \in \N_0$ the integrals
\begin{equation}\label{eq-Fourier}
f_n^c =\int_0^{2\pi} f(x)\cos nx\,dx,  \qquad f_n^s =\int_0^{2\pi} f(x)\sin nx\,dx. 
\end{equation}
Then   $\{f_{n}^c,f_n^s\}_{n\in\N_0} \subseteq l_{p'}$ and  there  exists a constant $C,$ depending  only on $p,$ such that
\begin{equation}\label{eq-YH1}
\Big (|f_0^c|^{p'}+ \sum_{n=1}^{\infty} (|f_n^c|^{p'} + |f_n^s|^{p'}) 
\Big )^{\frac{1}{p'}} \leq C \| f \|_{L^p(0,2\pi)}.
\end{equation}

Conversely, if $\{f_{n}^c,f_n^s\}_{n\in\N_0} \subseteq l_{p}, 1<p\leq 2$, then $f\in L^{p'}(0,2\pi), p'=p/(p-1)$ and  there  exists a constant $C',$ depending  only on $p,$ such that 
\begin{equation}\label{eq-RH2}
\| f \|_{L^{p'}(0,2\pi)} \leq C' \Big (|f_0^c|^p +\sum_{n=1}^{\infty} (|f_n^c|^p + |f_n^s|^p) \Big )^{\frac1p}.
\end{equation}
\end{theorem}

$x\in (0,2\pi)$ We consider the following systems of  functions defined on  $x\in (0,2\pi)$
\begin{equation} \label{eq-2-1}  
\{y_n\}_{n\in\N_0}=\{y_0^c=1,y_n^c(x)=\cos nx, y_n^s(x)= x \sin nx\}_{n\in \N},   
\end{equation}
\begin{equation} 
\begin{split}\label{eq-2-2}  
&\{\vartheta_n\}_{n\in\N_0}=
\{\vartheta^c_0,\vartheta^c_n,\vartheta^s_n\}_{n\in \N}, \\  
&\vartheta^c_0(x)=\frac{2\pi-x}{2\pi^2};  \quad \vartheta^c_n(x)=\frac{2\pi-x}{\pi^2} \cos nx; \quad  \vartheta^s_n(x)=\frac{1}{\pi^2} \sin nx.
\end{split}
\end{equation}

Direct calculations show that  they form {\it biorthonormal systems} in $L^2(0,2\pi).$  More precisely, for any $f\in L^2(0,2\pi)$ we define the linear continuous operator $b_\vartheta(\cdot)$ acting on  $L^2(0,2\pi)$ as  
$$
b_\vartheta(f)=(f;\vartheta)=\int_0^{2\pi} f(x)\vartheta(x)\, dx.
$$
The systems \eqref{eq-2-1}-\eqref{eq-2-2} are  called   {\it biorthonormal} iff
\begin{equation}\label{eq-2-1a}
\begin{split}
(y_k^c;\vartheta^c_n)=\delta_{kn}, &\qquad (y_k^s; \vartheta^s_n)=\delta_{kn},\\
 (y_k^s; \vartheta^c_n)=0, &\qquad 
(y_k^c; \vartheta^s_n)=0.
\end{split}
\end{equation}
Moreover, \eqref{eq-2-1} forms  a {\it   Riesz basis} in $L^2(0,2\pi)$  (cf.   \cite{6}), that is, we can expand any function $f\in L^2(0,2\pi)$ in {\it  biorthonormal  series }  of the form
\begin{equation}\label{eq-Riesz}
f(x)\sim (f;\vartheta^c_0) y^c_0(x) +\sum_{n=1}^\infty \Big((f;\vartheta_n^c) y_n^c(x)+
(f;\vartheta_n^s) y_n^s(x)\Big)
\end{equation}
where  $b_{\vartheta_n}\in (L^2(0,2\pi))^*.$

Our goal is to extend this theory to  the weighted  Lebesgue space      $L^p_\nu(0,2\pi)$ with $\nu\in A_p.$ 
\begin{theorem}\label{T2.1} 
 Let $\nu \in A_p(0,2\pi)$ with $ p\in(1,\infty).$ Then the system \eqref{eq-2-1}  forms a  basis in  $L^p_\nu(0,2\pi)$.
\end{theorem}
\begin{proof} 
First of all, we have to show that  \eqref{eq-2-1} is {\it minimal system}  and for this goal, it is sufficient to prove  that \eqref{eq-2-2}  is biorthonormal  system to \eqref{eq-2-1}  in $L^p_\nu(0,2\pi)$.
We can observe that  the  functional 
$
b_{\vartheta_n^c}(\cdot)=(\vartheta_n^c;\cdot), $  $  n\in \N_0
$ is uniformly  bounded in $L^p_\nu(0,2\pi).$
Indeed,  for any $f\in L^p_\nu(0,2\pi),$  using the uniform boundedness of the system $\{ \vartheta_n\},$ and  the H\"older inequality   we  obtain 
\begin{equation} \label{TM2.3} 
    \begin{split}
    | b_{\vartheta_n^c} (f)|&\leq \frac2\pi \int_0^{2\pi} |f(x)| \nu(x)^{\frac{1}{p}}\nu(x)^{-\frac{1}{p}}\,dx  \\
&\leq \frac2\pi \Big(\int_0^{2\pi}| f|(x)^p\nu(x)\,dx \Big)^{\frac{1}{p}} \Big(\int_J\nu(x)^{-\frac{p'}p}\,dx \Big)^{\frac{1}{p'}}  \\
&\leq 4\| f \|_{L^p_\nu(0,2\pi)}  [\nu]_p^{\frac1p}\|\nu\|_{L^1(0,2\pi)}^{-\frac1p}\leq C  \| f \|_{L^p_\nu(0,2\pi)}
\end{split}
\end{equation}
that implies  $b_{\vartheta_n^c} (\cdot)\in \big(L^p_\nu(0,2\pi)\big)^*$ for all $n \in \N_0$.  
Moreover, by  \eqref{eq-2-1a} we have 
\begin{equation*}
   b_{\vartheta^c_n}(y^c_k)=\delta_{kn},\quad 
   b_{\vartheta^c_n}(y^s_k)=0 \quad
   \forall \, k,n\in \N.
   \end{equation*}
   In the same way, we can see that  $b_{\vartheta^s_n}(\cdot) \in (L^p_\nu(0,2\pi))^*$, and it holds
\begin{equation*}  b_{\vartheta^s_n}(y^s_k)=\delta_{kn},\quad 
   b_{\vartheta^s_n}(y^c_k)=0,\quad
   \forall \, k,n\in \N.
\end{equation*}
This implies  biorthogonality of  \eqref{eq-2-2} and \eqref{eq-2-1} in $L^p_\nu(0,2\pi)$ and hence also  the minimality of \eqref{eq-2-1} in $L^p_\nu (0,2\pi)$. 

We need to show now that \eqref{eq-2-1} is also {\it complete} in $L^p_\nu(0,2\pi)$. By density arguments and  Lemma~\ref{L2.1}, it is sufficient to show that an arbitrary function from $C_0^{\infty}(0,2\pi)$ can be approximated by linear combination of functions belonging to   the system \eqref{eq-2-1}  in $L^p_\nu(0,2\pi)$.

Let $f\in C_0^\infty(0,2\pi)$ and consider 
$$
g(x)=\frac{2\pi-x}{\pi^2} f(x), \qquad  g\in C_0^{\infty}(0,2\pi).
$$
 For all $n \in \N$ we  consider the biorthogonal coefficients of $f$: 
\begin{equation}\begin{split}\label{eq-f}
      (f;\vartheta_0^c)= &\frac{1}{2\pi^2} \int_0^{2\pi} f(x)(2\pi-x)\, dx
    =\frac12(g;1)=\frac12 g_0 \\
     (f;\vartheta_n^c)= &\frac{1}{\pi^2} \int_0^{2\pi}f(x)\,(2\pi-x)\cos nx\,dx
    =(g;\cos nx)=g_n^c,,\\
       (f;\vartheta_n^s)=  &\frac{1}{\pi^2} \int_0^{2\pi}f(x) \sin nx\,dx
     =\frac1{\pi^2}(f;\sin nx)=\frac1{\pi^2} f_n^s.
     \end{split}
     \end{equation}
Integrating by parts twice and making use of the regularity of $f,$ we obtain the bounds
\begin{equation}\label{eq-ff} 
\begin{split}
 |(f;\vartheta_n^c)|\leq &\    \frac{1}{n^2}\int_0^{2\pi}\big| g''(x)\cos nx\big|\,dx\leq 
 \frac{c}{n^2}, \\[5pt]
    | (f;\vartheta_n^s)|\leq   &\  \frac{1}{n^2\pi^2}\int_0^{2\pi} \big| f''(x) \sin nx \big|\,dx \leq  \frac{c}{n^2}
\end{split}
\end{equation}
These estimates guaranty the total convergence of the biorthonormal  series 
 \begin{equation} \label{TM2.4} 
F(x)= (f;\vartheta_0^c)+\sum_{n=1}^{\infty}\Big((f;\vartheta_n^c)\cos nx+ (f;\vartheta_n^s) x\sin nx\Big ),
\end{equation}
and also the uniform convergence,  by the Weierstrass theorem. 
 According to the results of \cite{6}, the   system \eqref{eq-2-1} forms a basis in $L^2(0,2\pi)$ and hence $F=f.$  By Lemma~\ref{L2.1}, it follows also that \eqref{TM2.4} converges to $f$  in $L^p_\nu(0,2\pi)$ and  hence \eqref{eq-2-1} is a basis in $L^p_\nu(0,2\pi)$. Indeed,    considering  the  projectors
\begin{equation}
\begin{split}\label{def-Snm}
S_{n,m}(f)(x)= &\  \sum_{k=0}^n (f;\vartheta_k^c) y_k^c(x)+ \sum_{k=1}^m (f;\vartheta_n^s)
 y_k^s(x)\\
= &\  \frac{1}{2} g_0+\sum_{k=1}^n g_k^c \cos kx+\frac{x}{\pi^2} \sum_{k=1}^m f_k^s \sin kx.
\end{split}
\end{equation}
 Because of  the orthogonality of the trigonometric system  
 $L^p_\nu(0,2\pi)$-norm of  \eqref{def-Snm}
\begin{align*}
\| S_{n,m} (f)\|_{L^p_\nu (0,2\pi)} &\leq \Big \| \frac{1}{2}g_0 +  \sum _{k=1}^n g_k^c \cos kx \Big\|_{L^p_\nu(0,2\pi)}  + \frac2\pi\Big\|\sum_{k=1}^m  f_k^s\sin kx \Big\|_{L^p_\nu(0,2\pi)} 
    \end{align*}
where, in the first norm we have a partial sum of the  Fourier series for $g\in L^p_\nu(0,2\pi),$ while, in the second norm, we have the corresponding partial  sum for $f\in L^p_\nu(0,2\pi).$ 
Since the trigonometric system  forms a basis in $L^p_\nu (0,2\pi)$  iff $\nu \in A_{p}(0,2\pi)$ (cf. \cite{31}), we obtain  
    \begin{align*}
 \| S_{n,m} (f)\|_{L^p_\nu (0,2\pi)} 
    & \leq c \Big (\| g \|_{L^p_\nu(0,2\pi)} 
+  \|f\|_{L^p_\nu(0,2\pi)} \Big)\leq c\| f\|_{L^p_\nu(0,2\pi)}
\end{align*}
where we have used that $|g(x)|\leq C |f(x)|$ whether $x \in (0,2\pi)$ and the last estimate holds 
for all  $n,m \in \N$, with a constant  independent of $f.$ This implies that the projectors 
$\{S_{n,m} \}$ are uniformly bounded in $L^p_\nu(0,2\pi)$ and hence, the system \eqref{eq-2-1} forms a basis in $L^p_\nu(0,2\pi)$.
\end{proof}

\section{Solvability results}

Let us consider the following non-local problem for the Laplace equation, written  in formal way
\begin{align} \label{eq-3-1}
\begin{cases}
\Delta u(x,y)=0  &  \text{ for a.e. }  (x,y)\in \Pi \\
u(0,y)=u(2\pi,y) & \text{ for a.e. }  y\in(0,\infty) \\
u(x,0)=f(x) & \text{ for a.e. } x \in (0,2\pi) \\
u_x(0,y)= h(y) &  \text{ for a.e.  } y \in (0,\infty)
\end{cases}
\end{align}
where we initially  suppose that $f\in L^1(0,2\pi)$ and $h\in L^1 (\R_+).$

Under a {\it weighted strong solution} of \eqref{eq-3-1}, we mean a  function $u\in W_\nu^{2,p} (\overline{\Pi }),$  $\nu\in L^1(0,2\pi)$  verifying  the partial differential equation in  \eqref{eq-3-1} and the boundary conditions a.e.

For any $\xi>0$, we denote by  $\Pi_{\xi} = (0,2\pi) \times (0,\xi)$  a bounded rectangle in $\R^2$. Taking a test function $\varphi \in C_{J_0}^{\infty} (\overline{\Pi}),$    multiplying  the equation in \eqref{eq-3-1}, integrating over $\Pi,$  and applying   the Gauss-Ostrogradsky theorem we obtain
\begin{equation*}
\begin{split}
    0&=\iint_{\Pi }\Delta u(x,y)\varphi(x,y)\, dxdy =\iint_{\Pi_{ \xi_{\varphi}}}\Delta u(x,y)\varphi(x,y)\, dxdy \\
&=-\iint_{\Pi _{\xi_{\varphi}} }\nabla u(x,y)\nabla \varphi(x,y)\, dxdy
+ \int _{\partial \Pi _{\xi_{\varphi} } }\varphi(x,y)\, \frac{\partial u(x,y)}{\partial\bm  n}\, dl 
\end{split}
\end{equation*}
where ${\bm n}$ is the outer normal to $\partial \Pi_{\xi_\varphi}.$  
Applying the  boundary conditions of \eqref{eq-3-1}, taking into account that  on $J_0$ we have ${\bm n}=(-1,0),$ we obtain 
$$
\iint_{\Pi}\nabla u\nabla \varphi \,dxdy =-\int_0^{\infty }\varphi (0,y)h(y)\, dy, \qquad 
\forall \, \varphi \in C_{J_0}^{\infty}(\overline{\Pi}) .  
$$ 
This permits us to  give the following notion of weak solution.

\begin{definition}\label{D3.2} 
A function $u\in W^{1,p}_\nu (\Pi),$  $p\in(1,\infty)$ is a {\em weak solution} of  \eqref{eq-3-1} if it is differentiable  in distributional sense  and
\begin{equation} \label{eq-3-4}
\begin{cases}
\displaystyle \iint_{\Pi}\nabla u\nabla \varphi \,dxdy   +\int_0^{\infty}\varphi(0,y)h(y)\,dy=0\\ 
u(0,y)=u(2\pi, y)\quad y\in(0,\infty) \\
u(x,0)=f(x)\qquad\ x \in (0,2\pi)
\end{cases}
\end{equation} 
  for each $\varphi \in C_{J_0}^{\infty}(\overline{\Pi})$.
\end{definition}
\begin{theorem}\label{T3.1}  
Let $\nu \in A_p (0,2\pi)$, $1<p<\infty $, $f\in W_{\nu }^{1,p}(0,2\pi)$ such that $f(0)=f(2\pi )=0$ and $h\in L^1 (\R_+)$. If  problem \eqref{eq-3-4}  has a weak solution  $u\in W^{1,p}_\nu(\Pi)$ then it is unique. 
\end{theorem}
\begin{proof} 
In order to obtain uniqueness of the weak  solution of \eqref{eq-3-4} it is enough to prove the existence of only trivial solutions  of  the homogeneous problem  
\begin{align} \label{eq-3-5} 
\begin{cases}
\displaystyle \iint_{\Pi }\nabla u\, \nabla \varphi \, dxdy=0 & \forall \, \varphi \in C_{J_0}^{\infty}  (\overline{\Pi}) \\[5pt]
 u(0,y) = u(2\pi,y)  & y\in(0,\infty) \\[5pt]
 u(x,0)=0 &    x\in (0,2\pi).
\end{cases} 
\end{align} 
For any bounded domain $\Pi_{\xi}$  we have (as in  \eqref{TM2.3}) 
 $$
 \|u\|_{W^{1,1}(\Pi_{\xi })}\leq c   \|u\|_{W^{1,p}_\nu(\Pi_{\xi })}<\infty
 $$
with a constant $c=c(p,\|\nu\|_{L^1(0,2\pi)},[\nu]_p).$
 Then  $u$ has a trace  $u^{\xi} $ on  the upper boundary $J_{\xi} =\{(x,\xi ): \ x\in (0,2\pi)\}$ and by the density of   $C^{\infty } (\overline{{\Pi }}_{\xi })$  in $W^{1,1}(\Pi_{\xi })$  we have 
\begin{equation}\label{eq-3-4a}
\begin{split} 
&u^{\xi } (x):=u(x,\xi)  =\int _0^{\xi }\frac{\partial u(x,y)}{\partial y}\, dy,\\
 &u^0(x) =0, \qquad\qquad  \text{ for a.e. } x\in (0,2\pi).
 \end{split}
\end{equation}
It follows immediately  that  $u^\xi \in L^p_\nu (0,2\pi)$  and moreover the estimate
\begin{equation}\label{eq-xi}
\| u^\xi  \|_{L^p_\nu (0,2\pi)} \leq  c_{\xi} \| u\| _{W_{\nu }^{1,p}(\Pi )}
\end{equation}
is valid, with a constant  $c_{\xi } $   depending  only on $\xi $ and $p.$ 

In order to write 
the developing in series  of the solution of \eqref{eq-3-5} we calculate the   biorthonormal  coefficients of $u(x,y)$ as in  \eqref{eq-f}, that is  
\begin{equation}\label{eq-coeff}
\begin{split}
u_0^c(y)& =(u^y;\vartheta_0^c)=\frac{1}{2\pi^2 } \int _0^{2\pi } u^y (x)(2\pi -x)\, dx \\
u_n^c (y)& = (u^y;\vartheta_n^c)   =\frac{1}{\pi^2} \int _0^{2\pi }u^{y}(x)(2\pi -x)\cos nx\, dx,\\
u_n^s(y)& =(u^y;\vartheta_n^s)=\frac{1}{\pi^2} \int _0^{2\pi }u^y (x)\sin nx\, dx.
\end{split}
\end{equation}
Then the biorthonormal series of $u$  has the form
\begin{equation} \label{eq-3-11} 
u(x,y)\sim u_0^c(y)+\sum _{n=1}^{\infty}\Big (u_n^c(y)\cos nx+u_n^s(y)x\sin nx \Big ).  
\end{equation} 
It follows immediately, by \eqref{eq-xi} and the initial condition of \eqref{eq-3-5}, that $\|u^y\|_{L^p_\nu(0,2\pi)}$ vanishes  as $y\to 0^+,$ 
and  hence $u_0^c(0)=u_n^c(0)=u_n^s(0)=0.$ 

 For any  $\psi(y) \in C_0^{\infty } (\R_+ ),$   the functions  $\varphi_n (x,y)=\psi (y)\sin nx $  belong to $ C_{J_0}^{\infty } (\overline{\Pi}).$ Moreover, 
  $\frac{\partial u}{\partial y} \in L^1 (\Pi _{\xi } )$ then  (cf.    \cite{1}) 
the functions  $\{u_n^c,u_n^s \}_{n\in \N_0}$ are differentiable and 
$$
\frac{d u_n^s(y)}{d y} =\frac{1}{\pi^2} \int _0^{2\pi }\frac{\partial u(x,y)}{\partial y} \sin nx\, dx.
$$
Multiplying   both sides by $\psi'(y),$ integrating in $y$ over $\R_+,$ using the fact that $u$ solves   problem \eqref{eq-3-5} and taking into account that $\psi(0)=0$, we obtain
\begin{equation}
\begin{split}\label{calc1}
\int_0^{\infty }&\frac{d u_n^s(y)}{d y} \psi'(y)\, dy
=\frac{1}{\pi^2 } 
 \iint_{\Pi }\frac{\partial u(x,y)}{\partial y}  \psi'(y)\sin nx\, dxdy \\
=&\frac{1}{\pi^2} \iint _{\Pi}\frac{\partial u}{\partial y} 
\frac{\partial \varphi_n}{\partial y} \, dxdy
= -\frac{1}{\pi^2} \iint_{\Pi }\frac{\partial u}{\partial x} \, \frac{\partial \varphi_n}{\partial x}\, dxdy  \\
=&-\frac{n}{\pi^2} \iint _{\Pi }\frac{\partial u(x,y)}{\partial x} \psi (y)\cos nx \,dxdy\\
=&-\frac{n^2}{\pi^2 } \iint_{\Pi }u(x,y)\psi(y) \sin nx\, dxdy \\
=&-n^2\int _0^{\infty }u_n^s(y)\psi (y)\, dy.
\end{split}
\end{equation}

   In order to  estimate the $L^p$ norms of the coefficients  we calculate
  \begin{align*}
|u_n^s(y)|^p &\leq \frac{1}{\pi^{2p}} \Big (\int _0^{2\pi}|u(x,y)|\, dx \Big )^p\\
&=  \frac{1}{\pi^{2p}} \Big 
(\int_0^{2\pi}|u(x,y)|\nu(x)^{\frac{1}{p}}\nu(x)^{-\frac{1}{p}}\, dx \Big )^p \\
&\leq \frac{1}{\pi^{2p}} \Big  (\int_0^{2\pi} \nu(x)^{-\frac{p'}{p}} \,dx
\Big )^{\frac{p}{p'}} 
\int_0^{2\pi} |u(x,y)|^p \nu(x)\,  dx
\end{align*}
and hence
$$
 \| u_n^s \|_{L^p(\R_+)} \leq 
 c  \| u\|_{L^p_\nu (\Pi )} 
$$
with a constant $c=c(p,\|\nu\|_{L^1(0,2\pi)}, [\nu]_p).$
In a similar way,  we establish  that $\frac{d u_n^s}{dy} \in L^p (\R_+ ).$
Moreover, direct calculations show that the second derivative of $u_n^s$ exists in distributional sense. Precisely 
 \begin{align*}
\int_0^{\infty}&\frac{d^2u_n^s(y)}{dy^2}\psi(y)\,dy
 =\frac{1}{\pi^2} \iint_\Pi \frac{\partial^2u(x,y)}{\partial y^2}\psi(y)\sin nx\,dx dy\\
     &=-\frac{1}{\pi^2} 
\iint_\Pi   
   \frac{\partial u(x,y)}{\partial y} \psi'(y)\sin nx \,dxdy = n^2\int _0^{\infty }u_n^s(y)\psi (y)\, dy
 \end{align*}
  that implies $\frac{d^2 u_n^s}{dy^2} =n^2 u_n^s$  for a.e.  $y\in \R_+$  and   $u_n^s \in W^{2,p}(\R_+).$  Hence   for all $n\in \N$  the functions $u_n^s$  solve  
\begin{equation} 
\begin{cases}\label{eq-3-4b}
\dfrac{d^2 u_n^s(y)}{dy^2}=n^2 u_n^s(y) & \qquad \text{ for a.e. } y\in \R_+ \\[6pt]
u_n^s(0)=0 & \lim_{y\to\infty} |u(x,y)|=0.
\end{cases}
\end{equation}
Since we are looking for bounded solutions we consider general solution of the form $u_n^s(y)=a  e^{-ny}.$  It is easy to see that the initial condition in  \eqref{eq-3-4b}
 gives that $a=0,$ that is \eqref{eq-3-4b} has only  trivial solution and  hence    $u_n^s(y)=0.$

 In order to calculate  the   coefficients $u_n^c$ we use similar reasoning, defining the functions  
$$
\phi_0 (x,y)=\frac{1}{2\pi^2} \psi (y)(2\pi -x), \quad
\phi_n(x,y)=\frac{1}{\pi ^2} \psi (y)(2\pi -x)\cos nx
$$
for any  choice of $\psi \in C_0^{\infty} (\R_+).$ 
Calculations similar to those above and using the boundary conditions of \eqref{eq-3-5}  give
\begin{align*}
\int_0^{\infty }\frac{d u_0^c(y)}{dy} \psi'(y)\, dy &=\frac{1}{2\pi^2} \iint_{\Pi }
\frac{\partial u(x,y)}{\partial y} \psi '(y)(2\pi -x)\, dxdy \\
&=\iint_{\Pi} \frac{\partial u}{\partial y} \frac{\partial \Phi_0}{\partial y}\, dxdy = -\iint_{\Pi} \frac{\partial u}{\partial x} \frac{\partial \Phi_0}{\partial x}\, dxdy  \\
&=\frac{1}{2\pi^2} \int_0^{\infty} \Big( \int_0^{2\pi}\frac{\partial u(x,y)}{\partial x} \, dx\Big)\psi(y)\,dy=0. 
\end{align*}
 From the {\it Fundamental Lemma of  Calculus of Variations} it follows 
$$
\frac{d^2 u_0^c}{dy^2}=0\quad \Longrightarrow \quad  u_0^c(y)=ay+b.
$$
  Since  $u_0^c \in L^p (\R_+),$ then   necessarily we have that  $a=b=0.$ 

On the other hand, for all $n \in \N$ we have
\begin{equation} \label{eq-3-8}
\begin{split}
\int _0^{\infty}&
\frac{du_n^c(y)}{dy} \psi '(y)\, dy
=\frac{1}{\pi ^2} \iint_{\Pi }\frac{\partial u}{\partial y} (2\pi -x)\cos nx\, \psi '(y)\, dxdy  \\
&= \iint_{\Pi }\frac{\partial u}{\partial y}  \frac{\partial \phi_n }{\partial y} \, dxdy=-\iint_{\Pi }\frac{\partial u}{\partial x}  \frac{\partial \phi_n}{\partial x} \, dxdy  \\
&=\frac{1}{\pi ^2} \iint_{\Pi}\frac{\partial u}{\partial x}\Big (\cos nx+(2\pi -x)n\sin nx \Big ) \psi (y)\, dxdy.  
\end{split}
\end{equation}
From  the boundary conditions of \eqref{eq-3-5}, the definition of $\phi_n$ 
for all $n\in \N$ and $y\in \R_+,$ it follows 
\begin{equation}
\label{eq-3-8a}\begin{split}
&\frac{1}{\pi^2}
\int _0^{2\pi} \frac{\partial u(x,y)}{\partial x} \cos nx\, dx=n u_n^s(y)\\
&\frac{1}{\pi^2}\int _0^{2\pi}\frac{\partial u(x,y)}{\partial x} (2\pi -x)n\sin nx\,dx= n u_n^s(y)-n^2 u_n^c(y).
\end{split}
\end{equation}
Combining \eqref{eq-3-8} and \eqref{eq-3-8a}, keeping in mind that $u_n^s(y)=0$ for $y\in \R_+,$ and arguing as above we obtain 
$$
\int_0^{\infty}
\left(
\frac{d u_n^c(y)}{dy} \psi '(y) +\big(n^2 u_n^c(y) - nu_n^s(y)\big)\psi(y)\right)\,dy=0
$$
that means (integrating by parts only the first term)
$$
\int_0^{\infty}
\Big(-\frac{d^2 u_n^c(y)}{dy^2} +n^2 u_n^c(y) - nu_n^s(y)\Big) \psi(y)\,dy=0.
$$
From the {\it Fundamental Lemma of  Calculus of Variations}, it follows that $u_n^c$ solves the following Cauchy problem
\begin{equation} \label{eq-3-9}
\begin{cases}
\dfrac{d^2 u_n^c(y)}{dy^2}=n^2 u_n^c (y)  
& \text{ for a.e. } y\in \R_+\\[6pt] 
u_n^c(0)=0 &
\end{cases}
\end{equation} 
where we used that we already showed that $u_n^s=0.$ Hence  this problem  has only trivial solutions in $W^{2,p}(\R_+)$, that is  $u_n^c (y)=0$ for all $y\in \R_+$ and for any choice of $n\in \N$.

Consequently, all biorthonormal coefficients of $u^y(x)$   are equal to zero for a.e. $x\in (0,2\pi),$ and hence  $u(x,y)=0$ for a.e.   $(x,y)\in \Pi.$
 \end{proof}
 
Concerning  the question of existence of weak solution of  \eqref{eq-3-4} we study  the particular case when   $h=0.$

\begin{theorem}\label{T3.2}
Let $\nu \in A_p(0,2\pi)$ with $1<p<\infty$ and $f\in W_{\nu }^{1,p}(0,2\pi)$ such that $f(0)=f(2\pi)=0$. Then the problem
\begin{equation} \label{eq-3-4d}
\begin{cases}
\ds \iint_{\Pi}\nabla u\nabla \varphi \,dxdy=0 & \quad \forall \, \varphi \in C_{J_0}^{\infty}  (\overline{\Pi})\\ 
u(0,y)=u(2\pi, y)  &\quad  y\in(0,\infty) \\
u(x,0)=f(x)  &\quad  x \in (0,2\pi)
\end{cases}
\end{equation} 
has a unique  solution satisfying the estimate
$$
\| u\|_{W_{\nu }^{1,p}(\Pi )} \leq c\| f\|_{W_{\nu}^{1,p}(0,2\pi)}.
$$

\end{theorem}
\begin{proof}
Let $u\in W^{1,p}_\nu(\Pi)$  and 
$\{u_n^c(y),u_n^s(y)\}_{n\in \N_0}$  be the biorthonormal  coefficients of $u(x,y)$ with respect to  the system \eqref{eq-2-1}
and consider the  biorthonormal    series \eqref{eq-3-11} related to $u.$  Since  problem \eqref{eq-3-4d} is not homogeneous, we obtain 
$$
u^y(x)=\int_0^y \frac{\partial u(x,\tau)}{d\tau}\, du= u(x,y)-f(x) \quad \text{ for a.e. } x\in (0,2\pi).
$$
It is easy to see that  $u^y\in L^p_\nu(0,2\pi)$ for a.e. $y\in \R_+.$ Moreover
$$
\|u^y\|_{L^p_\nu(0,2\pi)} = \|u(\cdot, y)- f(\cdot)\|_{L^p_\nu(0,2\pi)}\to 0 \quad\text{ as } \  y\to 0^+.
$$

Our goal is to  prove that the series
\begin{equation} \label{eq-3-11a} 
u(x,y)\sim u_0^c(y)+\sum _{n=1}^{\infty}\Big (u_n^c(y)\cos nx+u_n^s(y)x\sin nx \Big )
\end{equation}
 with
$
u_n^c(y) =(u;\vartheta_n^c),  u_n^s(y) = (u;\vartheta_n^s) 
$
is a solution of \eqref{eq-3-5}. 
First of all, let us note that the total convergence in $L^p_\nu(\Pi)$ follows analogously as 
the convergence of \eqref{TM2.4}. 

Derivation  formally  \eqref{eq-3-11a} and arguing as above we obtain the following problems 
\begin{equation} \label{eq-3-12} 
\begin{cases} 
\dfrac{d^2 u_n^s(y)}{dy^2} =n^2 u_n^s (y) &\quad  y\in \R^+\\[5pt] 
u_n^s(0)=(f;\vartheta_n^s) &\quad n=1,2,\ldots.
\end{cases} 
\end{equation} 
Because of the boundedness of  $u_n^s$ as $y\to \infty,$ the unique solution of  problem \eqref{eq-3-12} is 
$
u_n^s(y)= (f;\vartheta_n^s) e^{-ny}.$

By similar arguments applied to 
$\{u_n^c\}$ we obtain 
\begin{equation} \label{eq-3-9a}
\begin{cases}
\dfrac{d^2 u_n^c(y)}{dy^2}=n^2 u_n^c (y)  - nu_n^s(y) 
& \text{ for a.e. } y\in \R_+\\[6pt] 
u_n^c(0)=(f;\vartheta_n^c) & \quad n=0,1,\ldots 
\end{cases}
\end{equation} 
If $n=0,$ it easily follows  that $ u_0^c=(f;\vartheta_0^c),$
 while for $n\in \N,$ the solutions vanishing for $y\to \infty$ are  
$$
u_n^c(y)=    (f;\vartheta_n^c)\,  e^{-ny} +\frac12 y(f;\vartheta_n^s)\, e^{-ny}
$$
and the series development of $u(x,y)$ becomes
\begin{equation} \label{eq-3-11b} 
\begin{split}
u(x,y)&\sim  (f;\vartheta_0^c) +\sum _{n=1}^{\infty}\big((f;\vartheta_n^c)+ \frac12 y(f;\vartheta_n^s)\big)\cos nx \, e^{-ny}\\[5pt]
&+ \sum _{n=1}^{\infty} (f;\vartheta_n^s) x \sin nx \, e^{-ny}.  
\end{split}
\end{equation}

Calculating formally the derivative  with respect to  $x,$ we obtain
\begin{equation} \label{eq-3-11c} 
\begin{split}
u_x(x,y)\sim&-\sum _{n=1}^{\infty}\big((f;\vartheta_n^c)+ \frac12 y(f;\vartheta_n^s)\big)n\sin nx \, e^{-ny}\\[5pt]
&+ \sum _{n=1}^{\infty} (f;\vartheta_n^s) \big(  \sin nx+xn\cos nx \big)\, e^{-ny}\\
&=\sum _{n=1}^{\infty} \big(  (f;\vartheta_n^s)-  (f;\vartheta_n^c)n- \frac12 y(f;\vartheta_n^s)n  \big)\sin nx \,e^{-ny}\\
&+\sum _{n=1}^{\infty} (f;\vartheta_n^s)xn\cos nx\, e^{-ny}=:v(x,y)+w(x,y)
\end{split}
\end{equation}
and we need to show the convergence of this series in $L^p_\nu(\Pi).$

 Let us start with $w(x,y),$ the series $v(x,y)$ can be treated in a similar manner. 
By \eqref{eq-Fourier} and \eqref{eq-coeff}, we have
$$
w(x,y)=\sum _{n=1}^{\infty} (f;\vartheta_n^s)xn\cos nx\, e^{-ny}=\frac{1}{\pi^2}\sum_{n=1}^{\infty}{f_n'}^c \, x\cos nx\, e^{-ny} .
$$ 
Let $\delta$ be as in Lemma~\ref{lem-RH} and take  $\alpha =1+\delta, $ then  $\alpha'=\frac{1+\delta}{\delta}$ is the conjugate of $\alpha.$ Applying the  H\"older inequality  for $ f \in L^p_\nu (0,2\pi)$ and Lemma~\ref{lem-RH} we obtain 
\begin{equation} \label{eq-3-16} 
\int _0^{2\pi}|f(x)|^p\nu (x)\,dx \leq C \Big (\int_0^{2\pi}|f(x)|^{p\alpha'}\, dx \Big )^{\frac{1}{\alpha'} }
\end{equation} 
where $C=(p, [\nu]_p, \|\nu\|_{L^1(0,2\pi)}).$  In order to estimate the norm of $w$ we consider the following two cases:
\begin{itemize}
\item  \fbox{$p\geq 2$} Hence   $p_1=p\alpha'>2,$ and $p_1'\in (1,2)$  and $p/p_1'>1.$ Then, by \eqref{eq-3-16}  for any  fixed $y \in \R_+$ we get
\begin{align*}
\| w(\cdot,y)\|^p_{L^p_\nu(0,2\pi)} \leq & C \Big (\int_0^{2\pi}|w(x,y)|^{p_1}\, dx 
\Big )^{\frac{p}{p_1}} \\
\leq & C \Big (\sum_{n=1}^{\infty}|{f'_n}^c|^{p_1'} e^{-np_1'y}\Big)^{\frac{p}{p'_1}} \leq C \sum_{n=1}^{\infty}|{f'_n}^c|^{p} e^{-npy}
\end{align*}
and the last inequality holds since $p/p'_1>1.$  
Integrating with respect to $y$ and using Holder inequality for sequence, we get

\begin{equation}\label{eq-3-17}
\begin{split}
\iint_\Pi& |w(x,y)|^p \nu (x) \,dxdy  \leq
C\sum_{n=1}^{\infty}|{f'_n}^c|^p\int_0^{\infty}e^{-npy}\,dy \\[5pt]
&= C\sum_{n=1}^{\infty}
\frac{|{f'_n}^c|^p}{n} \leq  C\Big (\sum_{n=1}^{\infty}\frac{1}{n^{\beta'}}\Big )^{\frac{1}{\beta'}} \Big (\sum_{n=1}^{\infty}|{f'_n}^c|^{p\beta}\Big)^{\frac{1}{\beta}} \leq C\| f' \|_{L^{(p\beta)'} (0,2\pi)}^p
\end{split}
\end{equation}
for any $\beta>1, $
where the constant $C$ depends on $\nu$ and $p.$

 Let $q\in(1,p)$ be as in Lemma~\ref{lem-RH} and take  $r=\frac{p}{q}.$ Then $1<r<p$ and by the H\"older inequality, we obtain
 \begin{equation}
\begin{split}\label{eq-3-18} 
\int_0^{2\pi} |f'|^r\, dx&=\int_0^{2\pi}|f'|^{\frac{p}{q}} \nu ^{\frac{1}{q}} \nu ^{-\frac{1}{q}}\, dx \\
&\leq \Big (\int_0^{2\pi}\nu ^{-\frac{1}{q-1}}\,dx \Big )^{\frac{q-1}{q}} \Big (\int_0^{2\pi} |f'|^p \nu \,dx \Big)^{\frac{1}{q}} \leq C \|f'\|^r_{L^p_\nu(0,2\pi)}
\end{split}
\end{equation}
where the constant depends on $[\nu]_q$ and $p.$

Chose  $\beta>1 $ such that $\beta_1=p\beta>p\geq 2$  and  $1<\beta_1'<r.$  Then, by the boundedness of $(0,2\pi),$ we have
\begin{equation}\label{eq-g}
\|f'\|_{L^{\beta_1'}(0,2\pi)} \leq c \| f'\|_{L^r(0,2\pi)} .
\end{equation}
  Making use of  \eqref{eq-3-17},   \eqref{eq-3-18}, and \eqref{eq-g},  we obtain
\begin{align*}
\| w\|_{L^p_\nu(\Pi)} &\leq C \| f'\|_{L^{\beta_1'}(0,2\pi)} \leq C\|f'\|_{L^r(0,2\pi)}\\
&\leq C \| f'\|_{L^p_\nu(0,2\pi)} \leq C \| f\|_{W^{1,p}_\nu(0,2\pi)}.
\end{align*}

\item \fbox{$1<p<2$} Taking $\alpha =1+\delta$ with   $0<\delta < p/(2-p),$ it follows that $p_1 =p\alpha'>2.$ Than we can  argue  as above  obtaining  boundedness of the $L^p_\nu$ norms via  the norm of $f.$
\end{itemize}

Considering  all the  series  that appear in the  expressions for $u, u_x$ and $u_y,$ we can estimate their norms, arguing  in a similar way. 
Unifying all  these estimates we obtain 
$$
\| u\|_{W_{\nu}^{1,p}(\Pi )} \leq C\| f\|_{W_{\nu}^{1,p}(0,2\pi)}
$$ 
with a constant independent of $f$. Moreover, by direct calculations,  we can check that  $u$ satisfies the differential equation in \eqref{eq-3-4d} in the weak sense.

What we need to show is that   $u$ satisfies also  the boundary conditions in \eqref{eq-3-4d}. Denote by $\theta_0$, $\theta_{2\pi}$  and $\theta_J$ the {\it trace operators} corresponding to the peaces  $J_0$, $J_{2\pi} $ and $J$ of the boundary $\partial \Pi, $ respectively. 

Let us show  that $\theta_Ju=f.$  Since  $u\in W^{1,1}(\Pi)$ then  $\theta_J u\in L^1(0,2\pi),$ hence we need to prove only that $\theta_J u=f$ a.e. on $J$. Consider, for all $m\in \N,$ the partial sums 
$$
u^m (x,y)=u_0^c(y)+\sum_{n=1}^m \Big (u_n^c(y)\cos nx+u_n^s(y)\,x\sin nx \Big ) \quad  (x,y)\in \Pi
$$
and
taking the trace on $J$ we obtain
\begin{align*}
\theta_J u^m(x)&=u^m(x,0)\\
&=f;(\vartheta_0^c) +\sum_{n=1}^m \Big ((f;\vartheta_n^c)  \cos nx+(f;\vartheta_n^s) x\sin nx \Big )=S_m(f)
\end{align*}
where  the last one  is the projector of  $f$ with respect to  system \eqref{eq-2-2}.
 Hence  by  Theorem~\ref{T2.1} we have
 $$
 \lim_{m \to \infty} \|S_m(f) -f\|_{L^p_\nu(0,2\pi)}   =   \|\theta_J u^m -f\|_{L^p_\nu(0,2\pi)}  = 0
$$
and   the convergence holds also with respect to the norm 
in $L^1(0,2\pi)$.
On the other hand, by the classical theory  
$$
 \lim_{m \to \infty}  \|\theta_J u^m - \theta_Ju\|_{L^1(0,2\pi)}  = 0
$$
  and hence  $\theta_J u=f$  a.e. on $J.$

It is easy to check  that  $u^m(0,y)=u^m (2\pi,y)$ for all $y>0$ and $m\in \N$ and arguing as above we can obtain that  $\theta_{J_0} u =\theta_{J_{2\pi}}u$.
\end{proof}
 The following result gives a necessary  condition under which the weak solution of \eqref{eq-3-4d} is also a strong one.

\begin{theorem}\label{T3.3} 
Let the conditions of Theorem~\ref{T3.2} hold. Then any function that is a  weak solution of \eqref{eq-3-4d}   and  belongs to $ W_{\nu }^{2,p}(\Pi)$       is a strong solution of 
\begin{equation}   \label{eq-3-1a}
\begin{cases}
\Delta u(x,y)=0  &  \text{\rm  for a.e. }  (x,y)\in \Pi \\
u(y)|_{J_0}=u(y)|_{J_{2\pi}} &  \text{\rm for a.e. }  y\in(0,\infty) \\
u(x,0)=f(x) & \text{\rm for a.e. } x \in (0,2\pi) \\
u_x(0,y)=0 &  \text{\rm for a.e.  } y \in (0,\infty).
\end{cases}
\end{equation}

\end{theorem}

\begin{proof}

For any $\eta(x) \in C_{2\pi }^{\infty}(\overline{J})$ and $\psi(y) \in C_0^{\infty} (\R_+),$  we consider the test function
$$
\varphi(x,y)=\psi (y)\eta (x) \in C_{J_0}^{\infty} (\Pi).
$$
Integrating by parts we obtain 
\begin{equation}\label{eq-3-18a}
\begin{split}
0=&\iint_{\Pi}\nabla u\,\nabla \varphi\,dxdy=\int_0^{\infty}\int_0^{2\pi}(u_x \varphi_x +u_y\varphi_y)\,dxdy\\
=&\int_0^{\infty} \psi(y)\left( \int_0^{2\pi} u_x \eta'(x)\,dx\right)\,dy + \int_0^{2\pi} \eta(x) \left(\int_0^{\infty}u_y \psi'(y)\, dy\right) dx\\
=& -\eta(0)\int_0^{\infty} \psi(y) u_x(0,y)\,dy -\int_\Pi \varphi \Delta u \, dxdy.
\end{split}
\end{equation}
Suppose that $\supp \psi(y)\subset [0,\xi]$ for some $\xi>0.$ Then we can write \eqref{eq-3-18a} in the form
\begin{equation} \label{eq-3-19} 
\iint_{\Pi_\xi}\varphi(x,y) \Delta u\,dxdy
=-\eta (0)\int_0^\xi u_x(0,y)\psi (y)\,dy.   
\end{equation}

Consider the following systems of functions 
\begin{align}\label{eq-3-19a}  
\{\eta_n(x)\}_{n\in\N_0}&=\{\eta^c_0=1, \,  \eta_n^c= \cos nx, \, 
 \eta_n^s=\sin nx\}_{n\in \N},\\[5pt]
\label{eq-3-19b}  
\{\psi_n(y)\}_{n\in\N_0}&=
\left\{\psi^c_0=1, \,\psi^c_n=\frac{\cos 2\pi ny}{\xi}, \, \psi^s_n=\frac{\sin 2\pi n y}{\xi}\right\}_{n\in \N} 
\end{align}
and the corresponding modified ones
\begin{align*}
\tilde{\eta}_n(x)&=x(2\pi -x)\eta_n (x)\in C^\infty_0 (\overline{J}) \\
\tilde{\psi }_n(y)&=y(\xi-y)\psi_n (y)\in C_0^\infty ([0,\xi]).
\end{align*} 
 Taking in \eqref{eq-3-18a}  the test function in the form
$$
\tilde\varphi_{mn}(x,y)=\tilde{\eta}_m(x)\tilde\psi_n(y)\in C_0^\infty(\Pi_\xi)
$$
for any $m,n\in\N_0,$ we obtain
\begin{equation} \label{eq-3-20} 
\int_0^{2\pi}\left( \int_0^\xi \tilde{\psi}_n(y)\Delta u\,dy \right) x(2\pi -x)\varphi (x)\,dx=0.     
\end{equation} 
 Since $u\in W_{\nu}^{2,p}(\Pi)$, it follows that  $\Delta u\in L^1(\Pi)$ and  
$$
F(x): =\int_0^\xi \tilde{\psi}_n(y) \Delta u\,dy\in L^1 (0,2\pi).
$$ 
Then by the Lebsgue  theorem,  \eqref{eq-3-20} implies that $F(x)=0$ for a.e. $x\in J$ and hence $\Delta u=0$ for a.e. $(x,y)\in \Pi_\xi.$
Because of the arbitrary of $\xi $ it follows  that
$
\Delta u=0
$ for 
a.e. $(x,y)\in \Pi$. Then by \eqref{eq-3-19}   we obtain
$$
\int_0^\infty u_x(0,y) {\psi}(y)\,dy=0
$$
for all $\psi \in C_0^\infty([0,\infty))$ and hence 
$
u_x(0,y)=0$  for a.e.  $  y>0.$
\end{proof}

\subsection*{Acknowledgments.} 
{\small
The research of B. Bilalov is supported by the Azerbaijan National Academy of Sciences (ANAS), Project Number: 19042020 and by the Science Development Foundation under the President of the Republic of Azerbaijan,
Grant No. EIF-BGM-4-RFTF1/2017- 21/02/1-M-19.

The authors 
S. Tramontano and L. Softova  are members of \textit{INDAM - GNAMPA}. 
}

\end{document}